\documentclass[preprint, 10pt, english]{elsarticle}
\usepackage{amsthm}
\usepackage{amsmath}
\usepackage{latexsym, amssymb}
\usepackage{txfonts}
\usepackage{mathtools}
\usepackage{mathtools}
\usepackage{color}
\usepackage{babel}
\usepackage[all]{xy}
\usepackage[capitalise]{cleveref}

\newtheorem{thm}{Theorem}[section] 

\newtheorem{lem}[thm]{Lemma}

\newtheorem{ques}[thm]{Question}

\theoremstyle{definition}

\newcommand\operA[2]{{\if!#2!\operatorname{#1}\else{\operatorname{#1}_{#2}^{\phantom{I}}}\fi}} 

\newcommand{\Trace}[1][]{\if!#1!\operatorname{Tr}\else{\operatorname{Tr}_{#1}^{\phantom{I}}}\fi} 

\long\def\forget#1\forgotten{{}} %

\def\({\left(}
\def\){\right)}

\newcommand\LAY[3][]{{\begin{array}{c}\mbox{#2} \if#1!{}\else{+}\fi \\ \mbox{#3}\end{array}}}

\makeatletter

\def\ps@pprintTitle{%
 \let\@oddhead\@empty
 \let\@evenhead\@empty
 \def\@oddfoot{}%
 \let\@evenfoot\@oddfoot}

\newcommand{\bigperp}{%
  \mathop{\mathpalette\bigp@rp\relax}%
  \displaylimits
}

\newcommand{\bigp@rp}[2]{%
  \vcenter{
    \m@th\hbox{\scalebox{\ifx#1\displaystyle2.1\else1.5\fi}{$#1\perp$}}
  }%
}
\makeatother

\newif\iffurther
\furtherfalse

\newcommand\kelly[1]{{\color{red}{#1}}}

\journal{??}

\begin{document}
\begin{frontmatter}

\title{Biquaternion Algebras, Chain Lemma and Symbol Length}

\author{Adam Chapman}
\address{School of Computer Science, Academic College of Tel-Aviv-Yaffo, Rabenu Yeruham St., P.O.B 8401 Yaffo, 6818211, Israel}
\ead{adam1chapman@yahoo.com}

\author{Kelly McKinnie}
\address{Department of Mathematical Sciences, University of Montana, Missoula, MT 59812, USA}
\ead{kelly.mckinnie@mso.umt.edu}

\begin{abstract}
In this note, we present a chain lemma for biquaternion algebras over fields of characteristic 2 in the style of the equivalent chain lemma by Sivatski in characteristic not 2, and conclude a bound on the symbol length of classes in ${_{2^n}Br}(F)$ whose symbol length in ${_{2^{n+1}}Br}(F)$ is at most 4.
 \end{abstract}

\end{frontmatter}

\section{Introduction}

A quaternion algebra is a central simple algebra of degree 2 over its center.
Given a field $F$ of characteristic 2, a quaternion algebra over $F$ takes the form $[\alpha,\beta)=F\langle i,j : i^2+i=\alpha, j^2=\beta, jij^{-1}=i+1\rangle$ for some $\alpha \in F$ and $\beta \in F^\times$.
These algebras generate ${_2Br}(F)$ by \cite[Theorem 9.1.1]{GS} (see also \cite[Chapter VII, Theorem 28]{Albert:1968}), and therefore the structure theory of their tensor products plays a major role in the study of this group.
One aspect of this theory is to understand the different symbol presentations of a single algebra by means of a chain lemma - a lemma that shows how one obtains from one presentation all the other presentations as well.
A three-step chain lemma for quaternion algebras (in characteristic 2) was given in \cite[Section 14, Theorem 7]{Draxl:1983}, and in \cite{CFM} it was used to bound the symbol length of central simple algebras of exponent $2^n$ that are Brauer equivalent to the tensor product of two symbol algebras of degree $2^{n+1}$ from above by 4. Upper bounds for the symbol length in ${_{2^n}Br}(F)$ when $\operatorname{char}(F)\neq 2$ where provided in \cite{Chapman:2022} for algebras that are Brauer equivalent to the tensor product of 3 or 4 symbol algebras of degree $2^{n+1}$ based on Sivatski's chain lemma for biquaternion algebras from \cite{Sivatski:2012}.

In \cite[Theorem 4.1]{Chapman:2015} a chain lemma for biquaternion algebras was introduced when $\operatorname{char}(F)= 2$, but this lemma does not specify the distance between two presentations and thus does not provide a concrete bound on the symbol length.
The goal of this paper is to provide a chain lemma in characteristic 2 in the manner of \cite{Sivatski:2012} and obtain analogous results to \cite{Chapman:2022}.

\section{Statement of Main Theorem}
Let $F$ be a field of characteristic 2. Let $A$ be a central simple algebra over $F$ of exponent 2 and degree 4. By a theorem of Albert, $A$ is a biquaternion algebra i.e., the tensor product of two quaternion algebras over $F$. In the spirit of \cite{Sivatski:2012} and other chain lemmas from quadratic form theory, this section investigates how far apart two symbol presentations of a biquaternion algebra can be. The notion of distance is formalized by the edges in the following graph.


Let $G_A=(V,E)$ be the labeled graph where $V$ is the set of ordered pairs of isomorphism classes of quaternion algebras $Q_1$ and $Q_2$ over $F$ with $A \cong Q_1\otimes_FQ_2$ and there is an edge between vertex $(Q_1,Q_2)$ and vertex $(Q_3,Q_4)$ if there exist symbol presentations $Q_1=[a,b)$ and $Q_2= [c,d)$ where either $Q_3=[a+c,b)$ and $Q_4=[c,bd)$ or $Q_3 = [a,b(a+c))$ and $Q_4 = [c,d(a+c))$. We label the first option as a type I edge, and the second as type II. Note that in both of these cases the isomorphism $Q_1\otimes Q_2 \cong Q_3\otimes Q_4$ is clear.

\begin{thm} For any biquaternion algebra $A$, the diameter of $G_A$ is at most 3, where the minimal path between two edges consists of at most two edges of type I and at most one edge of type II such that the type II edge is in between the type I edges if all three are needed.
\label{mainthm}
\end{thm}
Recall that the diameter of a graph is the supremum on the distance (i.e., length of minimal path) between two vertices, where the vertices range over all the pairs of vertices in the graph. In this case the diameter is measuring how far apart two representations of $A$ as a biquaternion algebra can be using the simple steps determining an edge. 

\section{Proof of Main Theorem}
 Before starting the formal proof of Theorem \ref{mainthm} we make some notes and remarks about biquaternion algebras and related quadratic forms.  For background see \cite{EKM}.

\begin{lem}\label{Albert}
If $\psi$ is a nonsingular form of dimension at least 2 and $\varphi=[1,\alpha] \perp \psi$ is isotropic, then there exist $r \in F$ and $v \in V_\psi$ for which
$r^2+r+\alpha+\psi(v)=0$.
\end{lem}

\begin{proof}
Since $\varphi$ is isotropic, there exist $x,y \in F$ and $w \in V_\psi$ such that $x^2+xy+\alpha y^2+\psi(w)=0$.
If $y\neq 0$, take $r=\frac{x}{y}$ and $v=\frac{1}{y} w$ and we are done.
Suppose then that $y=0$.
If $x=0$ then $\psi$ is isotropic, which means it is universal, and thus there exists $v\in V_\psi$ for which $\alpha+\psi(v)=0$.
Suppose $x\neq 0$.
Then $\alpha^2+\alpha+\alpha+\psi(\frac{\alpha}{x}w)=0$, so by taking $r=\alpha$ and $v=\frac{\alpha}{x}w$, the story is complete.
\end{proof}
 
Now, given $[a,b)\otimes [c,d)\otimes [\alpha,\beta)\otimes [\gamma,\delta) \cong M_{16}(F)$, set $\varphi$ to be the 10-dimensional quadratic form
\[\varphi = [a+c+\alpha+\delta,1]\perp b[1,a]\perp d[1,c]\perp \beta[1,\alpha]\perp\delta[1,\gamma] \in I_q^3F.\]
By \cite[pg.54]{Baeza:1978}, the Clifford Algebra of $\varphi$ is 
\[[a,b)\otimes [c,d)\otimes [\alpha,\beta)\otimes [\gamma,\delta)\otimes[a+c+\alpha+\gamma,1) \cong M_{32}(F).\]
The Arf invariant is $a+c+\alpha+\gamma+a+c+\alpha+\gamma = 0$ (see e.g., \cite[pg. 53]{Baeza:1978} or \cite[pg. 528]{MammoneShapiro:1989}).
Since $\varphi$ is 10-dimensional, has trivial Arf invariant and Clifford invariant, by \cite[Theorem 4.10]{DeMedts:2002} $\varphi$ is isotropic. This is the characteristic 2 version of (part of) Pfister's Theorem (see \cite[Page 123]{Pfister:1966}). For any $a \in F$, let $N_a = [1,a] = x^2+xy+ay^2$, the norm form for $a$. The following two lemmas are used in the proof of Theorem \ref{mainthm}.

\begin{lem} There exist $\lambda \in F$ and values of the forms $N_a,\,N_c,\,N_\alpha,\,N_\gamma$ so that
\begin{equation*}
a+c+\alpha+\gamma +\lambda+\lambda^2 = N_a\cdot b +N_c\cdot d+N_\alpha\cdot\beta+N_\gamma\cdot\delta
\end{equation*}
\label{lem1}
\end{lem}
\begin{proof} 
This is a special case of Lemma \ref{Albert}.
\end{proof}

\begin{lem} If $A = [a,b) = [\alpha,\beta)$ then there exists a square-central element $y \in A$ for which $A = [a,b) = [a,y^2) = [\alpha,y^2) = [\alpha,\beta)$.
\label{lem2}
\end{lem}
Note, this chain lemma is different from the one in Draxl's book \cite[$\S$14 Theorem 5]{Draxl:1983}, where the linking element is in the separable slot rather than the inseparable one.
\begin{proof}
Let $x$ and $z$ be elements in $A$ satisfying $x^2+x=a$ and $z^2+z=\alpha$. Consider the element $y=xz+zx$.
If $y=0$, then $x$ and $z$ commute, which means $z=x+\lambda$ for some $\lambda \in F$, and so $[a,b)=[\alpha,b)$ and the statement is trivial.
Suppose then that $y \neq 0$.
Then $xy+yx=zy+yz=y$ by a straight-forward computation, and so $[a,b)=[a,y^2)=[\alpha,y^2)=[\alpha,\beta)$.
\end{proof}

\noindent {\bf Proof of Theorem \ref{mainthm}}
Let $([a,b),[c,d))$ and $([\alpha,\beta),[\gamma,\delta))$ be two vertices in $G_A$. Then,
\begin{eqnarray*}
[a,b)\otimes[c,d) &=& [a,N_a\cdot b)\otimes[c,N_c\cdot d)\textrm{ \kelly{(same node)}}\\
&=&[a+N_a\cdot b,N_a\cdot b)\otimes[c+N_c\cdot d,N_c\cdot d)\textrm{ \kelly{(same node)}}\\
&=&[a+c+N_a\cdot b+N_c\cdot d,N_a\cdot b)\\
&&\otimes [c+N_c\cdot d,N_a\cdot b\cdot N_c\cdot d) \textrm{ \kelly{(one edge move of type I)}}\\
&=&[a+c+N_a\cdot b+N_c\cdot d+\lambda^2+\lambda,N_a\cdot b)\\
&&\otimes [c+N_c\cdot d,N_a\cdot b\cdot N_c\cdot d) \textrm{ \kelly{(adding $\lambda^2+\lambda$ keeps same node)}}\\
&=& [a',b')\otimes[c',d') \textrm{ \kelly{(same node, rename slots)}}
\end{eqnarray*}
Similarly,
\begin{eqnarray*}
[\alpha,\beta)\otimes[\gamma,\delta) &=& [\alpha,N_\alpha\cdot\beta)\otimes [\gamma,N_\gamma\cdot\delta)\textrm{ \kelly{(same node)}} \\
&=&[\alpha+N_\alpha\cdot\beta,N_\alpha\cdot\beta)\otimes [\gamma+N_\gamma\cdot\delta,N_\gamma\cdot\delta)\textrm{ \kelly{(same node)}}\\
&=&[\alpha+\gamma+N_\alpha\cdot\beta+N_\gamma\cdot\delta,N_\alpha\cdot\beta)\\
&&\otimes[\gamma+N_\gamma\cdot\beta,N_\alpha\cdot\beta\cdot N_\gamma\cdot\delta)\textrm{ \kelly{(one edge move type I)}}\\
&=& [a',\beta')\otimes[\gamma',\delta')\textrm{ \kelly{(same node, rename 3 of the 4 slots and reuse $a'$)}}
\end{eqnarray*}
The first slot in the last line above is $a'$ by Lemma \ref{lem1}. It is now enough to show one can get between two nodes of the form $([a,b),[c,d))$ and $([a,\beta),[\gamma,\delta))$ in one move (removing the $'$s for notational convenience). There are two cases to consider;

If $\gamma = c$ so that $[a,b)\otimes [c,d) = [a,\beta)\otimes[c,\delta)$ then $[a,b\beta) = [c,d\delta)$ and so $[a,b\beta) = [a,e)=[c,e) = [c,d\delta)$ for some $e \in F^\times$ by Lemma \ref{lem2}. Since $[a+c,e)$, $[a,b\beta e)$ and $[c,d \delta e)$ are split, by \cite[Corollary 4.7.5]{GS} we have $e = b\beta N_a = N_{a+c} = d\delta N_c$. 
Write $e$ as $e=x^2(a+c)+xy+y^2$ for some $x,y \in F$. If $x=0$ then $e=y^2$ is a square and 
\begin{eqnarray*}
[a,b)\otimes[c,d) &=& [a,be)\otimes[c,de) \textrm{ \kelly{(same vertex)}}\\
&=&[a,\beta)\otimes[c,\delta)\textrm{ \kelly{(because $e=b\beta N_a = d\delta N_c$, same vertex)}}
\end{eqnarray*}
Hence the distance between $(Q_1,Q_2)$ and $(Q_3,Q_4)$ in this case is 2. If $x \ne 0$, we can without loss of generality assume $x=1$ by dividing the equation through by $x^2$ so that $e = (a+c)+y+y^2$ and 
\begin{eqnarray*}
[a,b)\otimes[c,d) &=& [a+y+y^2,b)\otimes [c,d) \textrm{ \kelly{(same vertex)}}\\
&=& [a+y+y^2,be)\otimes[c,de) \textrm{ \kelly{(type II move)}}\\
&=&[a,b^2\beta N_a)\otimes[c,d^2\delta N_c)\textrm{ \kelly{(same vertex, expanded $e$)}}\\
&=& [a,\beta)\otimes[c,\delta) \textrm{ \kelly{(same vertex)}}
\end{eqnarray*}

This makes it a full 3 distance. We want to reduce to the case $c=\gamma$ as above. If $[a,b)\otimes[c,d) = [a,\beta)\otimes[\gamma,\delta)$ then $[a,b\beta) = [c,d)\otimes[\gamma,\delta)$.

The biquaternion algebra on the right-hand-side is not division, so its underlying Albert form is isotropic. Its Albert Form is $[1,c+\gamma] \perp d[1,c]\perp \delta [1,\gamma]$. Therefore, by Lemma \ref{Albert}, $\lambda^2+\lambda x+(c+\gamma)x^2+N_c d+N_\gamma\delta = 0$ implying $\lambda^2+\lambda+c+\gamma = N_c\cdot d+N_\gamma \cdot\delta$. 
Set $f=c+N_c\cdot d = \gamma+N_\gamma\delta+\lambda^2+\lambda$. Hence, 
 \begin{eqnarray*}
 [a,b)\otimes[c,d) &=& [a,b)\otimes[c,N_c\cdot d) \textrm{ \kelly{(same vertex)}} \\
 &=&[a,b)\otimes[c+N_c\cdot d,N_c\cdot d)\textrm{ \kelly{(same vertex)}}\\
 &=&[a,b)\otimes[f,N_c\cdot d)\textrm{ \kelly{(same vertex, relabeled to $f$)}}\\ \\
\left[a,\beta\right.)\otimes \left[\gamma,\delta\right.)&=&[a,\beta)\otimes[\gamma,N_\gamma\cdot \delta)\textrm{ \kelly{(same vertex)}} 
\\
&=&[a,\beta)\otimes[\gamma+N_\gamma \cdot \delta +\lambda^2+\lambda,N_\gamma\cdot \delta)\textrm{ \kelly{(same vertex)}} 
\\
&=&[a,\beta)\otimes[f,N_\gamma\cdot \delta)\textrm{ \kelly{(same vertex)}} 
  \end{eqnarray*}

Since $[a,b)\otimes[f,N_c\cdot d)$ and $[a,\beta)\otimes[f,N_\gamma\cdot \delta)$ are connected by at most one edge of type II, we are done.
\qed

It is left to ask the following question:
\begin{ques}
Can it be shown that, at least in some cases, the diameter is precisely 3 and not less?
\end{ques}
The natural candidate is the ``generic case" of $F=F_0(a,b,c,d)$ being the function field in 4 algebraically independent variables over some field $F_0$ of characteristic 2 (e.g., $\mathbb{F}_2$ or $\mathbb{F}_2^{sep}$) and $A=[a,b)\otimes [c,d)$. One should expect the vertices $([a,b),[c,d))$ and $([c,d),[a,b))$ to be a distance 3 from each other as in the analogous situation in \cite{Sivatski:2012}. Unfortunately, the valuation theory in characteristic 2 is much more complicated and therefore the argument in \cite{Sivatski:2012} cannot be easily adapted to our setting.

\section{Symbol Length Application}

Given a field $F$ of characteristic 2, a symbol algebra of degree $2^n$ over $F$ takes the form 
$[\omega,\beta)_{2^n,F}=F \langle \theta,j : \theta^2-\theta=\omega,  j\theta j^{-1}=\theta+(1,0,\dots,0) \rangle$
where $\omega$ is a truncated Witt vector of length $n$ over $F$, $\theta$ is a vector of variables that obeys the rules of Witt vectors when it comes to addition, $\theta^2$ is obtained by raising all the entries to the power of 2, and multiplication by a $j$ and $j^{-1}$ is slot-wise. See \cite{MammoneMerkurjev:1991} for further details.
These algebras generate ${_{2^n}Br}(F)$, and the symbol length of a class in ${_{2^n}Br}(F)$ is thus defined to be the minimal number of symbol algebras of degree $2^n$ required to express it as a tensor product.
Note that $[(\omega_1,\dots,\omega_n),\beta)_{2^n,F}^{\otimes 2^m}$ is Brauer equivalent to $[(\omega_1,\dots,\omega_{n-m}),\beta)_{2^{n-m},F}$.

Since ${_{2^m}Br}(F)$ embeds into ${_{2^n}Br}(F)$ for any $m<n$ in a natural way, one can ask what the symbol length of a class in ${_{2^m}Br}(F)$ is knowing its symbol length as a class in ${_{2^n}Br}(F)$.
In the special case of $m=n-1$, it is known that when the symbol length in ${_{2^n}Br}(F)$ is 1, then the symbol length in ${_{2^{n-1}}Br}(F)$ is at most 2 (\cite[Proposition 5]{MammoneMerkurjev:1991}), and when the symbol length in ${_{2^n}Br}(F)$ is 2, then the symbol length in ${_{2^{n-1}}Br}(F)$ is at most 4 (\cite[Corollary 5.5]{CFM}).
Here we consider the cases where the symbol length in ${_{2^n}Br}(F)$ is 3 or 4.
For the proofs, we will need the following lemma:

\begin{lem}\
\begin{enumerate}
\item $[\omega,\beta)_{2^n,F} \otimes [\pi,\delta)_{2^n,F}=[\omega+\pi,\beta)_{2^n,F} \otimes [\pi,\delta \beta^{-1})_{2^n,F}$
\item If $(t,0,\dots,0)=\omega+\pi$ as truncated Witt vectors, then $[\omega,\beta)_{2^n,F} \otimes [\pi,\delta)_{2^n,F}=[\omega,\beta t)_{2^n,F} \otimes [\pi,\delta t)_{2^n,F}$.
\end{enumerate}
\end{lem}

\begin{proof}
For 1, write $\theta,j$ and $\rho,k$ as the generators of $[\omega,\beta)_{2^n,F}$ and $[\pi,\delta)_{2^n,F}$.
It is easy to verify that $\theta+\rho$ and $j$ commute with both $\rho$ and $jk^{-1}$, the first two generate $[\omega+\pi,\beta)_{2^n,F}$ and the latter two generate $[\pi,\delta \beta^{-1})_{2^n,F}$ and we are done.

For 2, it is enough to note that $[(t,0,\dots,0),t)_{2^n,F}$ is trivial in the Brauer group, and on the other hand it is Brauer equivalent to $[\omega,t)_{2^n,F} \otimes [\pi,t)_{2^n,F}$.
\end{proof}

\begin{thm}
Let $A$ be a class in ${_{2^{n-1}}Br}(F)$ whose symbol length in ${_{2^n}Br}(F)$ is 4. Then the symbol length of $A$ in ${_{2^{n-1}}Br}(F)$ is at most 32.
\end{thm}

\begin{proof}
By assumption $A$ is Brauer equivalent to 
\[[\omega^1,\beta_1)_{2^n,F} \otimes [\omega^2,\beta_2)_{2^n,F}\otimes [\omega^3,\beta_3)_{2^n,F} \otimes [\omega^4,\beta_4)_{2^n,F},\] where $\omega^1,\omega^2,\omega^3,\omega^4$ denote different truncated Witt vectors ($\omega^i=(\omega_1^i,\dots,\omega_n^i)$) rather than powers of a given vector.
Since $A$ is in ${_{2^{n-1}}Br}(F)$, $[\omega^1_1,\beta_1)_{2,F} \otimes [\omega^2_1,\beta_2)_{2,F}\otimes [\omega^3_1,\beta_3)_{2,F} \otimes [\omega^4_1,\beta_4)_{2,F}=0$, and therefore $[\omega^1_1,\beta_1)_{2,F} \otimes [\omega^2_1,\beta_2)_{2,F}= [\omega^3_1,\beta_3)_{2,F} \otimes [\omega^4_1,\beta_4)_{2,F}$.
By the chain lemma we have just proved (Theorem \ref{mainthm}), these two symbol presentations of the same biquaternion algebra are connected by a path of 3 steps - up to two steps of type I and at most one step of type II in between.
We shall proceed under the assumption that the full path is required in order to move from one presentation to the other.
In the case of a shorter path, a similar argument applies, with a smaller bound on the symbol length.
Therefore, there exist $a_1,a_2,a_3,c_1,c_2,c_3 \in F$ and $b_1,b_2,b_3,d_1,d_2,d_3 \in F^\times$ such that
\begin{itemize}
\item  $[\omega^1_1,\beta_1)_{2,F}=[a_1,b_1)_{2,F}$, $[\omega^2_1,\beta_2)_{2,F}=[c_1,d_1)_{2,F}$\textrm{ \kelly{(same vertex)}}
\item $[a_1+c_1,b_1)_{2,F}=[a_2,b_2)_{2,F}$ and $[c_1,b_1d_1)_{2,F}=[c_2,d_2)_{2,F}$\textrm{ \kelly{(type I)}}
\item $[a_2,b_2(a_2+c_2))_{2,F}=[a_3,b_3)_{2,F}$ and $[c_2,d_2(a_2+c_2))_{2,F}=[c_3,d_3)_{2,F}$\textrm{ \kelly{(type II)}}
\item $[a_3+c_3,b_3)_{2,F}=[\omega^3_1,\beta_3)_{2,F}$ and $[c_3,d_3b_3)_{2,F}=[\omega^4_1,\beta_4)_{2,F}$\textrm{ \kelly{(type I)}}
\end{itemize}
Let $e_2,\dots,e_n \in F$ be the unique elements for which $(a_2,0,\dots,0)+(c_2,e_2,\dots,e_n)=(a_2+c_2,0,\dots,0)$.
Consider now the following four classes
\begin{align}
&[\omega^1,\beta_1)_{2^n,F} \otimes [\omega^2,\beta_2)_{2^n,F} \otimes \left([(a_1,0,\dots,0),b_1)_{2^n,F} \otimes [(c_1,0,\dots,0),d_1)_{2^n,F}\right)^{op}\\
&\begin{multlined}[t]
[(a_1,0,\dots,0)+(c_1,0,\dots,0),b_1)_{2^n,F} \otimes [(c_1,0,\dots,0),d_1b_1^{-1})_{2^n,F} \otimes \\
\hspace{1.5in}{\left([(a_2,0,\dots,0),b_2)_{2^n,F} \otimes [(c_2,e_2,\dots,e_n),d_2)_{2^n,F}\right)^{op}}
\end{multlined}\\
&\begin{multlined}[t]
[(a_2,0,\dots,0),b_2(a_2+c_2))_{2^n,F} \otimes [(c_2,e_2,\dots,e_n),d_2(a_2+c_2))_{2^n,F}\otimes \\
\hspace{1.5in}\left([(a_3,0,\dots,0),b_3)_{2^n,F} \otimes 
[(c_3,0,\dots,0),d_3)_{2^n,F}\right)^{op}
\end{multlined}\\
&[(a_3,0,\dots,0),b_3)_{2^n,F} \otimes [(c_3,0,\dots,0),d_3)_{2^n,F} \otimes [\omega^3,\beta_3)_{2^n,F} \otimes [\omega^4,\beta_4)_{2^n,F}
\end{align}
Each of these classes is of symbol length at most 8 in ${_{2^{n-1}}Br}(F)$ (because each such class can be written as the product of two classes, where each class is of exponent dividing $2^{n-1}$ a product of two symbols of degree $2^n$), and their product is $A$. Therefore, the symbol length of $A$ in ${_{2^{n-1}}Br}(F)$ is at most 32.
\end{proof}

\begin{thm}
Let $A$ be a class in ${_{2^{n-1}}Br}(F)$ whose symbol length in ${_{2^n}Br}(F)$ is 3. Then the symbol length of $A$ in ${_{2^{n-1}}Br}(F)$ is at most 9.
\end{thm}

\begin{proof}
By the assumption $A$ is Brauer equivalent to $[\omega,\beta_1)_{2^n,F} \otimes [\pi,\beta_2)_{2^n,F}\otimes [\rho,\beta_3)_{2^n,F}$.
Then $[\omega_1,\beta_1)_{2,F} \otimes [\pi_1,\beta_2)_{2,F}$ is Brauer equivalent to $[\rho_1,\beta_3)_{2,F}$. Therefore the underlying Albert form of the biquaternion algebra is isotropic, which means that $\omega_1+\pi_1+\lambda^2+\lambda+\beta_1(\omega_1 x^2+xy+y^2)+\beta_2(\pi_1 t^2+tz+z^2)$ for some $\lambda,x,y,z,t \in F$.
We continue under the assumption that $\gamma_1=\omega_1 x^2+xy+y^2$ and $\gamma_2=\pi_1 t^2+tz+z^2$ are nonzero. A similar argument applies when one of them is zero.
Consider the classes
$$[\omega,\beta_1)_{2^n,F} \otimes [\pi,\beta_2)_{2^n,F} \otimes \left([\omega,\gamma_1)_{2^n,F} \otimes [\pi,\gamma_2)_{2^n,F}\right)^{op}$$
$$[\omega+(\gamma_1,0,\dots,0)+\pi+(\gamma_2,0,\dots,0),\gamma_1)_{2^n,F}$$
$$[\pi+(\gamma_2,0,\dots,0),\gamma_2\gamma_1^{-1})_{2^n,F}\otimes [\rho,\beta_3)_{2^n,F}$$
The first class is of symbol length at most 4 in ${_{2^{n-1}}Br}(F)$ because of \cite[Corollary 5.5]{CFM}, the second term is of symbol length 1 in that group by \cite[Lemma 5.4]{CFM}, and the last term is of symbol length at most 4. Altogether, the tensor product of these three classes, which is Brauer equivalent to $A$, is of symbol length at most 9.
\end{proof}

\section*{Declarations}

Ethical approval, competing interests, funding and availability of data and materials are not applicable to this work. The authors were equally involved in also stages of preparation.

\bibliographystyle{alpha}

\end{document}